\documentclass[a4paper,12pt, twoside, reqno]{amsart}
\usepackage{amssymb}
\usepackage{amsmath}
\newtheorem{theorem}{Theorem}
\newtheorem{corollary}[theorem]{Corollary}
\newtheorem{lem}{Lemma}

\newtheorem{defn}{Definition}
\newtheorem{ex}{Example}
\newtheorem{remark}{Remark}[section]

\title{Approximating fixed points of enriched nonexpansive mappings by Krasnoselskij iteration in Hilbert spaces}
\author{Vasile BERINDE}
\begin{document}
\maketitle \pagestyle{myheadings} \markboth{Vasile Berinde} {Approximating fixed points of enriched ...}
\begin{abstract}
Using the technique of enrichment of contractive type mappings by Krasnoselskij averaging, presented here for the first time, we introduce and study the class of {\it enriched nonexpansive mappings}  in Hilbert spaces. In order to approximate the fixed points of enriched nonexpansive mappings we use the Krasnoselskij iteration for which we prove strong and weak convergence theorems. Examples to illustrate the richness of the new class of contractive mappings  are also given. 

Our results in this paper extend some classical convergence theorems established by Browder and Petryshyn in [Browder, F. E., Petryshyn, W. V., {\it Construction of fixed points of nonlinear mappings in Hilbert space}, J. Math. Anal. Appl. {\bf 20} (1967), 197--228.] from the case of nonexpansive mappings to that of enriched nonexpansive mappings, thus including many other important related results from literature as particular cases.
\end{abstract}

\section{Introduction}

Let $K$ be a nonempty subset of a real normed linear space $X$. A map $T:K\rightarrow K$ is called nonexpansive if 
$$
\|Tx-Ty\|\leq \|x-y\|, \forall x,y\in K.
$$
An element $x\in K$ is said to be a {\it fixed point} of $T$ is $Tx=x$. Denote by $Fix\,(T)$ the set of all fixed points of $T$.  

Nonexpansive mappings are a limit case of Picard-Banach contractions, i.e., of mappings $T:K\rightarrow K$ which satisfy, in the case of normed linear spaces,  a contraction condition of the form
$$
\|Tx-Ty\|\leq c\cdot \|x-y\|, \forall x,y\in K,
$$
where $c\in [0,1)$ is the contraction coefficient. 

Banach contraction mapping principle, see for example \cite{Ber07}, assures that  any contraction $T$ on a Banach space has a unique fixed point which is the limit of the sequence of its iterates $\{T^n x_0\}$, for any $x_0\in K$. A similar assertion for nonexpansive mappings is not more true. 

Indeed, if $K$ is a closed nonempty subset of a Banach space $X$ and $T:K\rightarrow K$ is nonexpansive, it is known that $T$ may not have a fixed point or it may have many fixed points, and third, it may may happen that, even if $T$ has a unique fixed point, the Picard iteration $\{x_n=T^n x_0\}$  may fail to converge to
such a fixed point. One of the simplest examples of such a map is $Tx=1-x$ on $[0,1]$ with the usual norm, which gives, for $x_0=1$ say, $x_{2n}=1$ and $x_{2n+1}=0$. Also, rotation about the origin of the unit disk in the plane is another example of nonexpansive mapping having a unique fixed point while $\{x_n=T^n x_0\} (x_0\neq 0)$ does not converge. 

These aspects made the study of nonexpansive mappings one the major and most active research areas of nonlinear analysis since the mid-1960's. We mention the early contributions in this respect, due to Mann \cite{Mann} in 1953 and to Krasnoselskij \cite{Kra55} in 1955, respectively, who considered instead of Picard iteration (which does not converge, in general, for nonexpansive mappings) an explicit averaged iteration of the form $x_{n+1}=f(x_n,T x_n), n\geq 0$. We state in the following the classical  result due to Krasnoselskij \cite{Kra55}.
\begin{theorem} \label{tK}
Let $X$ be a uniformly convex Banach space and $C$ a closed subset of $X$. If $T:C \rightarrow  C$ is nonexpansive  
and $\overline{T(C)}$ is compact, then the mapping defined by
$
T_{\frac{1}{2}} x = \frac{1}{2}x + \frac{1}{2} Tx
$
has the property that its sequence of iterates always converges to a
point of $T$.
Since $T$ and $T_{\frac{1}{2}}$ have the same fixed points, the limit of a convergent sequence given by
$
x_{n+1}= \frac{1}{2} x_n+\frac{1}{2}T x_n, n\geq 0
$
 is necessarily a fixed point of $T$.
\end{theorem}

The following two problems were important in the study of nonexpansive mappings and were approached, amongst others, by \cite{Sch},  \cite{Bro65}-\cite{Bro65c}, \cite{Ede66}, \cite{Hal}, \cite{Dot} etc. :

a) what additional conditions on the structure of the ambient space $X$ and / or on the properties of $T$ must be added to assure that a nonexpansive mapping has at least one fixed point ?

b) how one can locate and approximate such a fixed point ?

Apart from the theoretical aspects mentioned above, nonexpansive mappings are extremely important from the point of view of their applications because, see \cite{Bru}:
  
"1) Nonexpansive maps are intimately connected with the monotonicity methods developed since the early 1960Õs and constitute one of the first classes of nonlinear mappings for which fixed point theorems were obtained by using the fine geometric properties of the underlying Banach spaces instead of compactness properties.

2) Nonexpansive mappings appear in applications as the transition operators for initial value problems of differential inclusions of the form
$0 \in du +T(t)u$, where the operators $\{T(t)\}$ are, in general, set-valued and are accretive or dissipative and minimally continuous."

A simple search in MathScinet shows the impressive number of 3606 indexed papers that bear the term "nonexpansive" in their title and 5826 papers that hold the term "nonexpansive" anywhere, as by  November 30, 2018.

Having in view the significant interest of researchers for the study of nonexpansive mappings, the main aim of the present paper is to introduce a  larger class of mappings of nonexpansive type, which will be called {\it enriched nonexpansive mappings}, and to answer in the affirmative the two problems a) and b) mentioned above in that context.

It is our personal belief that the concept of {\it enriched nonexpansive mapping} introduced here for the first time as well as the technique of enrichment of contractive mappings itself, by means of which we obtained the results presented in this paper and in \cite{BerP19}-\cite{Pac19b}, will open new and productive research avenues in the field of nonlinear analysis.

\section{Enriched nonexpansive mappings in Hilbert spaces} 

\begin{defn}\label{def0}
Let $(X,\|\cdot\|)$ be a linear normed space. A mapping $T:X\rightarrow X$ is said to be an {\it enriched nonexpansive mapping} if there exists $b\in[0,ü\infty)$  such that
\begin{equation} \label{eq3}
\|b(x-y)+Tx-Ty\|\leq (b+1) \|x-y\|,\forall x,y \in X.
\end{equation}
To indicate the constant involved in \eqref{eq3} we shall also call $T$ as a  $b$-{\it enriched nonexpansive mapping}. 
\end{defn}

\begin{remark}
It is easy to see that any nonexpansive mapping $T$ is a $0$-enriched mapping, i.e., it satisfies \eqref{eq3} with $b=0$.  

We note that, according to Theorem 12.1 in \cite{Goe}, in a Hilbert space any  enriched nonexpansive mapping which is also firmly nonexpansive is nonexpansive.

It is very important to note that, similar to the case of nonexpansive mappings, any enriched nonexpansive mapping is {\it continuous}.
\end{remark}

\begin{ex}[ ] \label{ex1}
\indent

Let $X=\left[\dfrac{1}{2},2\right]$ be endowed with the usual norm and $T:X\rightarrow X$ be defined by $Tx=\dfrac{1}{x}$, for all $x\in \left[\dfrac{1}{2},2\right]$. Then 

(i) $T$ is Lipschitz continuous with Lipschitz constant $L=4$ and $T$ is not nonexpansive;

(ii) $T$ is  a $3/2$-enriched nonexpansive mapping.

(iii)  $Fix\,(T)=\{1\}$ but $T$ is not quasi-nonexpansive.

\begin{proof}
 (i) Assume $T$ is nonexpansive. Then 
$$
|Tx-Ty|\leq |x-y|,\forall x,y \in X,
$$
which, for $x=1$ and $y=1/2$, leads to the contradiction $1\leq 1/2$. 

(ii) The enriched nonexpansive condition \eqref{eq3} reduces in this case to 
$$
\left|b(x-y)+\frac{1}{x}-\frac{1}{y}\right|\leq (b+1) |x-y| \Leftrightarrow \left|b-\frac{1}{xy}\right|\cdot |x-y|\leq (b+1) \cdot |x-y|.
$$
It easy to check that, for any $b\geq 3/2$, we have
$$
\left|b-\frac{1}{xy}\right|\leq b+1, \forall x,y \in \left[\dfrac{1}{2},2\right],
$$
which proves that $T$ is  a $3/2$-enriched nonexpansive mapping. 

(iii) Assume $T$ is quasinonexpansive. Then we must have
$$
\left|\frac{1}{x}-1\right |\leq |x-1|, \forall x \in \left[\dfrac{1}{2},2\right].
$$
Just take $x=\dfrac{1}{2}$ to reach the contradiction $1\leq \dfrac{1}{2}$.
\end{proof}
\end{ex}


We need some definitions and results for proving our main results.

\begin{defn} \cite{Pet} \label{def1}
Let $H$ be a Hilbert space and $C$ a
subset of $H$. A mapping $T:C\rightarrow H$ is called \textit{demicompact} 
if it has the property that whenever $\{u_{n}\}$ is a bounded sequence in $H$
and $\{Tu_{n}-u_{n}\}$ is strongly convergent, then there exists a
subsequence $\{u_{n_{k}}\}$ of $\{u_{n}\}$ which is strongly convergent.
\end{defn}

\begin{defn} \cite{BroP67} \label{def2}
Let $H$ be a Hilbert space and $C$ a closed convex
subset of $H$. A mapping $T:C\rightarrow C$ is called \textit{asymptotically regular} (on $C$)
if, for each $x\in C$, 
$$
\|T^{n+1}x-T^n x\|\rightarrow 0  \textnormal{ as } n\rightarrow \infty.
$$ 
\end{defn}

The following Lemma, which is adapted after Corollary to Theorem 5 in \cite{BroP67} will be used in the proof of the main result of this section.
\begin{lem} \label{lem1}
Let $H$ be a Hilbert space and $C$ a closed convex
subset of $H$. If the mapping $U:C\rightarrow C$ is nonexpansive and $Fix\,(U)\neq \emptyset$ then, for any given $\lambda\in (0,1)$, the mapping $U_\lambda=I+(1-\lambda) U$ maps $C$ into $C$, has the same fixed points as $U$ and is asymptotically regular.
\end{lem}

Now we can state and prove the main result of this section. 
\begin{theorem}  \label{th1}
Let $C$ be a bounded closed convex
subset of a Hilbert space $H$ and  $T:C\rightarrow C$ be a $b$-enriched nonexpansive and demicompact mapping. Then the set $Fix\,(T)$ of fixed points of $T$ is a nonempty convex set and there exists $\lambda\in \left(0,1\right)$ such that, for any given $x_0\in C$,  the Krasnoselskij iteration $\{x_n\}_{n=0}^{\infty}$ given by
\begin{equation} \label{eq4}
x_{n+1}=(1-\lambda) x_n+\lambda Tx_n,\,n\geq 0,
\end{equation}
converges strongly to a fixed point of $T$.
\end{theorem}

\begin{proof}
Since $T$ is enriched nonexpansive, by Definition \ref{def0} it follows that there exists a constant $b$, $b\in[0,ü\infty)$, such that
$$
\|b(x-y)+Tx-Ty\|\leq (b+1) \|x-y\|,\forall x,y \in C.
$$
By putting $b=\dfrac{1}{\mu}-1$, it follows that $\mu\in (0,1]$ and the previous inequality is equivalent to
\begin{equation} \label{eq5}
\|(1-\mu)(x-y)+\mu Tx-\mu Ty\|\leq  \|x-y\|,\forall x,y \in C.
\end{equation}
Denote $T_\mu x=(1-\mu) x+\mu T x$. Then inequality \eqref{eq5} expresses the fact that 
$$
\|T_\mu x-T_\mu y\|\leq  \|x-y\|,\forall x,y \in C,
$$
i.e., that the averaged operator $T_\mu$ is nonexpansive. 

By means of Browder-Goede-Kirk fixed point theorem (e.g., Theorem 4 in \cite{BroP67}), reproduced as Theorem 3.1 in \cite{Ber07}, it follows that $T_\mu$ has at least one fixed point. 

Note also that, in view of Lemma \ref{lem1}, $Fix\,(T)=Fix\,(T_\mu)\neq \emptyset$.

Since $H$ is a Hilbert space, $Fix\,(T)$ is also convex, see for example Theorem 6 in \cite{BroP67}. 
 We include here a proof of the fact that $Fix\,(T_\mu)$ is convex, i.e., when $x,y\in Fix\,(T_\mu)$ and $\lambda\in[0,1]$, we have
$$
u_\lambda=(1-\lambda)x+\lambda y\in Fix\,(T_\mu).
$$
Indeed, since $T_\mu$ is nonexpansive, we have
\begin{equation} \label{eq6}
\|T_\mu u_\lambda- x\|=\|T_\mu u_\lambda-T_\mu x\|\leq \|u_\lambda- x\|
\end{equation}
and, similarly,
\begin{equation} \label{eq7a}
\|T_\mu u_\lambda- y\|\leq \|u_\lambda- y\|.
\end{equation}
Now, by \eqref{eq6} and \eqref{eq7a}, we have
$$
\|x-y\|\leq \|x-T_\mu u_\lambda\|+\|T_\mu u_\lambda-y\|\leq \|u_\lambda- x\|+\|u_\lambda- y\|=\|\lambda (x-y)\|
$$
$$
+\|(1-\lambda)(x- y)\|=\|x-y\| \Rightarrow \|x-T_\mu u_\lambda\|+\|T_\mu u_\lambda-y\|=\|x-y\|.
$$
The last equality implies the existence of some nonnegative numbers $a,b$ with $a,b\leq 1$, such that
$$
x-T_\mu u_\lambda=a(x-u_\lambda)
$$
and 
$$
y-T_\mu u_\lambda=b(y-u_\lambda).
$$
Hence, for all $\lambda\in[0,1]$,
\begin{equation} \label{eq8u}
\|x-y\|=\|x-T_\mu u_\lambda\|+\|T_\mu u_\lambda-y\|=a\cdot \|x-u_\lambda\|+b\cdot \|y-u_\lambda\|.
\end{equation}
Now, just take $\lambda=1$ in \eqref{eq8u} to get $a=1$ and then take $\lambda=0$ in \eqref{eq8u} to get $b=1$. 

This shows that 
$T_\mu u_\lambda= u_\lambda$, that is, $u_\lambda\in Fix\,(T_\mu)$. So $Fix\,(T_\mu)$ is convex and hence $Fix\,(T)$, is convex, too.
Thus, the first part of our theorem is proven. 

In order to prove the last part of the theorem, consider the sequence $\{x_n\}_{n=0}^{\infty}$ given by 
$$
x_{n+1}=(1-\lambda) x_n+\lambda T_\mu x_n,\,n\geq 0.
$$
It is obvious that $\{x_n\}_{n=0}^{\infty}$ lies in $C$ and hence it is bounded.

Denote $U_\lambda=(1-\lambda) I+\lambda T_\mu$, where $I$ is the identity map. Then, since $T_\mu$ is nonexpansive, by Lemma \ref{lem1} it follows that $U_\lambda$ is asymptotically regular, i.e.,
$$
\|x_n-U_\lambda x_{n}\|\rightarrow 0,\, \textnormal{ as } n\rightarrow \infty.
$$

We also have
\begin{equation} \label{eq8a}
U_\lambda x-x=\lambda (T_\mu x-x)=\lambda \mu (Tx-x),
\end{equation}
and hence
$$
\|x_n-T_\mu x_{n}\|\rightarrow 0,\, \textnormal{ as } n\rightarrow \infty.
$$
Since, by hypothesis, $T$ is demicompact, it follows by \eqref{eq8a} that $T_\mu$
 is demicompact, too. Hence, there exists a subsequence $\{x_{n_k}\}$ of $\{x_n\}_{n=0}^{\infty}$ which converges strongly in $C$.
 Denote
 $$
 \lim_{k\rightarrow \infty} x_{n_k}=q.
 $$
 Then, by the continuity of $T_\mu$ it follows that
 $$
 T_\mu x_{n_k} \rightarrow T_\mu q,\, \textnormal{ as } k\rightarrow \infty.
 $$
 Therefore, $\{x_{n_k} -T_\mu x_{n_k} \}$ converges strongly to $0$ and simultaneously, $\{x_{n_k} -T_\mu x_{n_k} \}$ converges strongly to $q-T_\mu q$, which proves that $q=T_\mu q$, i.e., $$q\in Fix\,(T_\mu)=Fix\,(T).$$

The convergence  of the entire sequence $\{x_n\}_{n=0}^{\infty}$ to $q$ now follows from the inequality
$$
\|x_{n+1}-q\|\leq \|x_{n}-q\|, \,n\geq 0,
$$
which is a direct consequence of the nonexpansivity of $U_\lambda$ (which, in turn, is a consequence of the nonexpansivity of $T_\mu$). 

Hence, for any $x_0\in C$, the Krasnoselskij iteration $\{x_n\}_{n=0}^{\infty}$, given by
$$
x_{n+1}=U_\lambda x_n=(1-\lambda) x_n+\lambda T_\mu x_n=(1-\lambda) x_n+\lambda [(1-\mu)x_n+\mu T x_n]
$$
$$
=(1-\lambda) x_n+\lambda [(1-\mu)x_n+\mu T x_n]
$$
$$
=(1-\lambda \mu) x_n+\lambda \mu Tx_n
$$
converges strongly to $q\in Fix\,(T)$ as $n\rightarrow \infty$.

To get exactly the formula \eqref{eq4} for the Krasnoselskij iteration $\{x_n\}_{n=0}^{\infty}$ given above, just simply denote $\lambda:=\lambda \mu\in (0,1)$.
\end{proof}

\begin{remark}\label{rem1}
Theorem \ref{th1} is an extension of Lemma 3 of Petryshyn \cite{Pet} and of its global variant (Theorem 6) in Browder and Petryshyn \cite{BroP67}, by considering instead of nonexpansive mappings the larger class of enriched nonexpansive mappings.
\end{remark}

\begin{remark}\label{rem2}
The class of demicompact operators contains, among many classes of operators (see \cite{Pet}), the compact operators and, in particular, 
the completely continuous and strongly continuous operators.

Hence, from Theorem  \ref{th1} (and also from Corollary \ref{cor1}) one obtains the pioneering result of Krasnoselskij from 1955 (\cite{Kra55}, \cite{Kra56})
$$
\frac{1}{2}(x_n+Tx_n)\rightarrow q\in Fix\,(T) \,(\textnormal{ as } n\rightarrow \infty
$$
which was subsequently extended by Schaefer in 1957 (\cite{Sch}) to the general Krasnoselskij scheme 
$$
(1-\lambda)x_n+\lambda Tx_n\rightarrow q\in Fix\,(T), \,(\textnormal{ as } n\rightarrow \infty,\, 0<\lambda <1),
$$
a result established in the general setting of a uniformly Banach space.

The above results for compact operators have been extended to strictly convex Banach spaces by Edelstein in \cite{Ede66}.

\end{remark}

\begin{corollary} (Theorem 6, \cite{BroP67}) \label{cor1}
Let $C$ be a bounded closed convex
subset of a Hilbert space $H$ and  $T:C\rightarrow C$ be a  nonexpansive and demicompact operator. Then the set $Fix\,(T)$ of fixed points of $T$ is a nonempty convex set and there exists $\lambda\in \left(0,1\right)$ such that, for any given $x_0\in C$, the Krasnoselskij iteration $\{x_n\}_{n=0}^{\infty}$ given by
$$
x_{n+1}=(1-\lambda) x_n+\lambda Tx_n,\,n\geq 0,
$$
converges strongly to a fixed point of $T$.
\end{corollary}

\begin{proof}
Any nonexpansive mapping is a $0$-enriched nonexpansive mapping. Hence, Corollary \ref{cor1} follows from Theorem \ref{th1} for $b=0$ , that is, for $\mu=1$.
\end{proof}

Remind that a mapping $T:C\rightarrow C$  is called {\it generalized pseudocontractive} with constant $r$, see \cite{Ver}, \cite{Ber02} and \cite{Ber07}, if 
\begin{equation}\label{vv}
\|Tx-Ty\|^2\leq r \|x-y\|^2+\|Tx-Ty-r(x-y)\|^2,\,x,y\in C.
\end{equation}
It easy to prove that, in a Hilbert space, condition \eqref{vv} is equivalent to
\begin{equation}\label{vu}
\langle Tx-Ty, x-y\rangle\leq r \|x-y\|^2,\,x,y\in C.
\end{equation}

Note also that the function $T$ in Example \ref{ex1}  is generalized pseudocontractive with constant $r>0$ arbitrary and Lipschitzian with constant $L=4$. 

Theorem \ref{th1} is also an extension of the main result in \cite{Ver}, see also \cite{Ber02} and \cite{Ber07}.

\begin{corollary} (\cite{Ver}) \label{cor2}
Let $C$ be a nonempty closed convex
subset of a Hilbert space $H$ and  $T:C\rightarrow C$ be a  Lipschitzian and generalized pseudocontractive operator (with constants $s$ and $r$, respectively, $r<1$). Then  for any given $x_0\in C$ and any fixed number $\lambda$, $0<\lambda<\frac{2(1-r)}{1-2 r+s^2}$, the Krasnoselskij iteration $\{x_n\}_{n=0}^{\infty}$ given by
\begin{equation}\label{w}
x_{n+1}=(1-\lambda) x_n+\lambda Tx_n,\,n\geq 0,
\end{equation}
converges strongly to the unique fixed point of $T$.
\end{corollary}

\begin{proof}
Note that, in a Hilbert space, the $b$-enriched nonexpansive condition  \eqref{eq3} is equivalent to
\begin{equation}\label{uv}
2b\cdot \langle Tx-Ty, x-y\rangle+1\cdot \|Tx-Ty\|^2\leq (2b+1)\cdot \|x-y\|^2.
\end{equation}
Since $T$ is a Lipschitzian and generalized pseudocontractive operator we have
$$
2b\cdot \langle Tx-Ty, x-y\rangle+1\cdot \|Tx-Ty\|^2\leq (2rb+s^2) \|x-y\|^2
$$
and so, in order to have \eqref{uv} satisfied, it suffices to take $b\geq \dfrac{s^2-1}{2(1-r)}$, which is always possible. 

On the other hand, $T$ is demicompact, which is a consequence of the fact that $T$ is Lipschitzian, hence continuous, and of the generalized pseudocontractivity property, according to Lemma 2 in \cite{Pet}, where are presented some classeses of demicontractive operators.

Now the conclusion of the corollary follows by Theorem \ref{th1}. 

It is important to note that the uniqueness of the fixed point is in this case a direct consequence of the generalized pseudocontractivity property. Indeed, suppose that $T$ possesses two distinct fixed points, i.e., $p,q\in Fix\,(T)$ and $p\neq q$. Then, by \eqref{vu}, we get
$$
\langle x-y, x-y\rangle \leq  r \|x-y\|^2 \Leftrightarrow  \|x-y\|^2\leq r \|x-y\|^2,
$$
a contradiction, since $r<1$.
\end{proof}

We close this section by reminding that, under the assumptions of Corollary \ref{cor2}, if additionally, one has
$$
r\leq s,
$$
where $r$ and $s$ are the constant of generalised pseudocontractivity and Lipschitz constant, respectively, then it is possible to identify the fastest Krasnoselskij iterations amongst those given by \eqref{w}, see \cite{Ber02}, \cite{Ber02a} and \cite{Ber07}, which is obtained for
$$
\lambda=\frac{1-r}{1-2 r+s^2}.
$$
This result has been at the origin of introducing a notion of rate of convergence for comparing two fixed point iterations in \cite{Ber04a}, a concept that turn out to be very useful, see also \cite{Ber16} and references therein.

\section{Weak convergence theorems for enriched nonexpansive mappings in Hilbert spaces} 

A key tool in proving the main result in the previous section has been the fact that the averaged map of a nonexpansive mapping $T$ is asymptotically regular, a property which has been used in conjunction with the demicompactness of $T$. 

In the absence of the demicompactness property of the operator $T$, the asymptotic regularity alone does not imply in general the strong convergence of the Krasnoselskij sequence $\{x_n\}_{n=0}^{\infty}$  but the weak convergence can be still assured, as shown by the next theorems which extend Theorems 7 and 8 in Browder and Petryshyn \cite{BroP67} from nonexpansive mappings to enriched nonexpansive mappings.

\begin{theorem}  \label{th2}
Let $C$ be a bounded closed convex
subset of a Hilbert space $H$ and  $T:C\rightarrow C$ be an enriched nonexpansive operator with $Fix\,(T)=\{p\}$. Then, for any given $x_0\in C$ and any fixed number $\lambda$, $0<\lambda<1$, the Krasnoselskij iteration $\{x_n\}_{n=0}^{\infty}$ given by
\begin{equation} \label{eq4a}
x_{n+1}=(1-\lambda) x_n+\lambda Tx_n,\,n\geq 0,
\end{equation}
converges weakly to $p$.
\end{theorem}

\begin{proof}
We use similar arguments to those in the proof of Theorem \ref{th1} to show that $T_\mu x=(1-\mu) x+\mu T x$, which map $C$ into $C$  is nonexpansive. Note also that, in view of Lemma \ref{lem1}, $Fix\,(T_\mu) =Fix\,(T)=\{p\}$.

To prove the theorem, it suffices to show that if  $\{x_{n_j}\}_{j=0}^{\infty}$ given by 
$$
x_{n_j+1}=(1-\lambda) x_{n_j}+\lambda T_\mu x_{n_j},\,j\geq 0,
$$
converges weakly to a certain $p_0$, then $p_0$ is a fixed point of $T_\mu$ (and of $U_\lambda=(1-\lambda) I+\lambda T_\mu$) and hence of $T$ and therefore $p_0=p$. 

Suppose that $\{x_{n_{j}}\}_{j=0}^{\infty}$  does not converge weakly to $p$. 

Using the same arguments like in the proof of Theorem \ref{th1}, we obtain that $U_\lambda$ is nonexpansive and asymptotically regular, that is,
$$
\|x_{n_j}-T_\mu x_{n_j}\|\rightarrow 0,\, \textnormal{ as } j\rightarrow \infty.
$$
On the other hand
$$
\|x_{n_j}-U_\lambda p_0\|\leq \|U_\lambda x_{n_j}-U_\lambda p_0\|+\|x_{n_j}-U_\lambda x_{n_j}\|
$$
$$
\leq \|x_{n_j}-p_0\|+\|x_{n_j}-U_\lambda x_{n_j}\|,
$$
which implies that
\begin{equation} \label{eq7}
\limsup \left(\|x_{n_j}-U_\lambda p_0\|-\|x_{n_j}-p_0\| \right)\leq 0.
\end{equation}
Similarly to the proof of Theorem \ref{th1}, we have
$$
\|x_{n_j}-U_\lambda p_0\|^2=\|(x_{n_j}-p_0)+(p_0-U_\lambda p_0\|^2
$$
$$
=\|x_{n_j}-p_0\|^2+\|p_0-U_\lambda p_0\|^2+2\langle x_{n_j}-p_0,p_0-U_\lambda p_0\rangle,
$$
which, together with the fact that $\{x_{n_j}\}$ converges weakly to $p_0$, implies
\begin{equation} \label{eq8}
\lim_{j\rightarrow \infty} \left[\|x_{n_j}-U_\lambda p_0\|^2-\|x_{n_j}-p_0\|^2\right]=\|p_0-U_\lambda p_0\|^2.
\end{equation}
We also have
$$
\|x_{n_j}-U_\lambda p_0\|^2-\|x_{n_j}-p_0\|^2=\left(\|x_{n_j}-U_\lambda p_0\|-\|x_{n_j}-p_0\|\right)\cdot
$$
\begin{equation} \label{eq9}
\left(\|x_{n_j}-U_\lambda p_0\|+\|x_{n_j}-p_0\|\right).
\end{equation}
Since $C$ is bounded, the sequence $\{\|x_{n_j}-U_\lambda p_0\|+\|x_{n_j}-p_0\|\}$ is bounded, too, and therefore by combining \eqref{eq7}, \eqref{eq8} and \eqref{eq9}, we get
$$
\|p_0-U_\lambda p_0\|=0,
$$
that is,
$$
U_\lambda p_0=p_0,
$$
which implies
$$
 p_0\in Fix\,(U_\lambda)=Fix\,(T_\mu)=Fix\,(T)=\{p\}.
$$
\end{proof}

\begin{remark}\label{rem3}
1) Theorem \ref{th2} is an extension of Theorem 7 in Browder and Petryshyn \cite{BroP67}, see also Theorem 3.3 in \cite{Ber07}, by considering enriched nonexpansive mappings instead of nonexpansive mappings.

2) The assumption $Fix\,(T)=\{p\}$ in Theorem \ref{th2} may be removed to obtain the following more general result.
\end{remark}

\begin{theorem}  \label{th3}
Let $C$ be a bounded closed convex
subset of a Hilbert space $H$ and  $T:C\rightarrow C$ be an enriched nonexpansive operator. Then, for any given $x_0\in C$ and any fixed number $\lambda$, $0<\lambda<1$, the Krasnoselskij iteration $\{x_n\}_{n=0}^{\infty}$ given by \eqref{eq4a}
converges weakly to a fixed point of $T$.
\end{theorem}

\begin{proof}
By the arguments in the proof of Theorem \ref{th1}, we have  $Fix\,(T)=Fix\,(T_\mu)\neq \emptyset$, where, as usually,  $T_\mu x=(1-\mu) x+\mu T x$. Moreover, by Theorem \ref{th1}, $Fix\,(T_\mu)$ is convex. Since $T_\mu$ is nonexpansive, for any $p\in Fix\,(T_\mu)$ and for each $n$ we have
$$
\|x_{n+1}-p\|\leq \|x_{n}-p\|,\,n\geq 0,
$$
which shows that the function
$$
g(p)=\lim_{n\rightarrow \infty} \|x_{n}-p\|,\,p\in Fix\,(T_\mu),
$$
is well defined and is a lower semicontinuous convex function on $Fix\,(T_\mu)$. 

The rest of the proof is similar to that of Theorem 8 in Browder and Petryshyn \cite{BroP67} and therefore is omitted.
\end{proof}

\section{Conclusions and further study}

In this paper  we introduced and studied the class of {\it enriched nonexpansive mappings}  in the setting of a Hilbert space $H$. We have shown that any enriched nonexpansive mapping defined on a bounded, closed and convex subset $C$ of $H$ has  fixed points in $C$ and that the set of all its fixed points is convex. In order to approximate a fixed point of an enriched nonexpansive mapping,  we used the Krasnoselskij iteration for which we have proven a strong convergence result (Theorem \ref{th1}) as well as two weak convergence theorems (Theorems \ref{th2} and \ref{th3}). 

We illustrated the richness of the new class of mappings  by means of Examples \ref{ex1} and Corollary \ref{cor2}. We conclude that  all nonexpansive mappings are included in the class of enriched nonexpansive mappings, which is also independent of that of quasinonexpansive mappings (and which includes all nonexpansive mappings possessing fixed points). Moreover, see the proof of Corollary \ref{cor2}, all Lipschitz and generalized pseudocontractive mappings  are also included in the class of {\it enriched nonexpansive mappings}. Note also that, similarly to the case of nonexpansive mappings, any enriched nonexpansive mapping is continuous. 

Our results extend some convergence theorems in \cite{BroP67} from nonexpansive mappings to enriched nonexpansive mappings and thus include many other important related results from literature as particular cases, see \cite{Ber02}, \cite{Ber02a}, \cite{Ber07}, \cite{BroP66}, \cite{Chi}, \cite{Kra55}, \cite{Pet}, \cite{Sch} etc. 

For some other old and recent related developments related to nonexpansive type mappings we refer to \cite{Ag}-\cite{BerP09},  \cite{Bro65}-\cite{Xu16} and references therein. 

As mentioned in the Introduction, the study of nonexpansive mappings attracted a large number of researchers who contributed significantly to the development of this area of research. As a consequence, various single-valued and multi-valued self and nonself mappings related to nonexpansive mappings were introduced and studied independently of in connection to nonexpansive mappings (see MathScinet or  zbMATH): quasi-nonexpansive mappings (see \cite{Dot}), strictly quasi-nonexpansive mappings (see \cite{Tri}), firmly nonexpansive mappings (see \cite{Bru73}), asymptotically nonexpansive mappings (see \cite{Goe72}), asymptotically quasi-nonexpansive mappings, generalized asymptotically nonexpansive mappings, generalized nonexpansive mappings (see \cite{Goe73}), $\alpha$-nonexpansive mappings (see \cite{Ao11}), Suzuki nonexpansive mappings (\cite{Suz}), Suzuki generalized nonexpansive mappings (\cite{Fal}), pseudocontractive  mappings and strictly $k$-pseudocontractive  mappings (see \cite{Bro65c}), $k$-demicontractive mappings (\cite{Mar73}, \cite{Mar77}, \cite{Hic}), Bregman nonexpansive mappings, Bregman strongly nonexpansive mappings (see \cite{Reich10}), Bregman relatively nonexpansive mappings, affine nonexpansive mappings, weakly nonexpansive mappings (\cite{BerK}), Berinde nonexpansive mappings (\cite{Bun18}-\cite{Bus18}), Pre\v si\' c nonexpansive mappings (\cite{BerP09}), Pre\v si\' c-Kannan nonexpansive mappings (\cite{Fuk}), nearly nonexpansive mappings (see \cite{Sahu}), $G$-nonexpansive mappings, iterated nonexpansive mappings (see \cite{Ben}), $I$-nonexpansive mappings, $Q$-nonexpansive mappings, $\Phi$-nonexpansive mappings, ($L$)-type mappings (see \cite{Llo}), non-spreading mappings (see \cite{Koh}), hybrid mappings (see \cite{Tak10}), $\lambda$-hybrid mappings (see \cite{Ao}) etc.

So, a challenging problem would be to establish the relationships between the class of enriched nonexpansive mappings, on the one side, and most of the classes of nonexpansive-type mappings mentioned above, on the other side, excepting of course nonexpansive mappings and quasi-nonexpansive mappings, whose relations to enriched nonexpansive mappings have been already established.

\vskip 0.5 cm {\it Department of Mathematics and Computer Science

North University Center at Baia Mare

Technical University of Cluj-Napoca 

Victoriei 76, 430122 Baia Mare ROMANIA

E-mail: vberinde@cunbm.utcluj.ro}

\vskip 0.5 cm {\it Academy of Romanian Scientists  (www.aosr.ro)

E-mail: vasile.berinde@gmail.com}

\end{document}